\documentclass{bmcart}

\usepackage{graphicx}
\usepackage{epsfig, graphicx}
\usepackage{subfigure}
\usepackage{latexsym,amsfonts,amsbsy,amssymb,amsmath,amsthm}
\usepackage{color}
\usepackage{longtable}
\usepackage{enumerate}
\usepackage{algorithm2e}
\usepackage{amsthm}
\newtheorem{thm}{Theorem}[section]
\newtheorem{lem}{Lemma}[section]
\newtheorem{cor}{Corollary}[section]
\newtheorem{remark}{Remark}[section]

\newcommand{\set}[1]{\left\{#1\right\}}
\renewcommand{\mid}{\mathrel{}\middle|\mathrel{}}
\newcommand{\norm}[1]{\left\lVert#1\right\rVert} 
\newcommand{\abs}[1]{\left\lvert#1\right\rvert}
\newcommand{\FM}{{\scriptsize F\!M}}

\usepackage{url}

\usepackage{lipsum}
\usepackage{amsfonts}
\usepackage{graphicx}
\usepackage{epstopdf}

\startlocaldefs
\endlocaldefs

\begin{document}

\begin{frontmatter}

\begin{fmbox}
\dochead{Research}


\title{Error constant estimation under the maximum norm for linear Lagrange interpolation}


\author[
      addressref={aff1},
      email={shirleymaegalindo@cmu.edu.ph}   
]{\inits{SM}\fnm{Shirley Mae} \snm{Galindo}}
\author[
   addressref={aff3},
   email={k-ike@m.sc.niigata-u.ac.jp}
]{\inits{K}\fnm{Koichiro} \snm{Ike}}
\author[
   addressref={aff3},
   corref={aff3},
   email={xfliu@math.sc.niigata-u.ac.jp}
]{\inits{X}\fnm{Xuefeng} \snm{Liu}}

\address[id=aff1]{
  \orgname{Graduate School of Science and Technology, Niigata University}, 
  \city{Niigata},                              
  \cny{Japan}                                    
}
\address[id=aff3]{
  \orgname{Faculty of Science, Niigata University}, 
  \city{Niigata},                              
  \cny{Japan}                                    
}

\begin{artnotes}
\end{artnotes}

\end{fmbox}



\begin{abstractbox}

\begin{abstract} 

For the Lagrange interpolation over a triangular domain, we propose an efficient algorithm  to rigorously evaluate the interpolation error constant under the maximum norm by using the finite element method (FEM). In solving the optimization problem corresponding to the interpolation error constant, the maximum norm in the constraint condition is the most difficult part to process. To handle this difficulty, a novel method is proposed by combining the orthogonality of the interpolation associated to the Fujino--Morley FEM space and the convex-hull property of the Bernstein representation of functions in the FEM space. Numerical results for the lower and upper bounds of the interpolation error constant for triangles of various types are presented to verify the efficiency of the proposed method.


\end{abstract}




\end{abstractbox}
%

\end{frontmatter}

\section{Introduction}\label{header-n457}

In this paper, we consider the error estimation for the linear Lagrange interpolation over triangle elements and provide explicit values for the error constant in the error estimation under the $L^\infty$-norm.

Before the detailed discussion of our results, let us introduce  the existing literature on the Lagrange interpolation function in a general scope.

\begin{itemize}
    \item {\bf (1D case)} Given $1$-dimensional interval $I=(0,1)$,
    since $H^1(I)\subset C(\overline{I})$,  we can define the Lagrange interpolation $\Pi^L u$ such that $\Pi^L u$ is a linear function satisfying
$(u - \Pi^L u)(0) =(u - \Pi^L u)(1) =0$.
Then, the following results are well known as optimal estimates if $u$ is regular enough in the sense that the right-hand sides of the inequalities are well defined:
\begin{eqnarray*}
&\norm{u-\Pi^L u}_{0,I} \le \frac{1}{\pi^2} \abs{u}_{2,I},\quad \abs{u-\Pi^L u}_{1,I} \le \frac{1}{\pi} \abs{u}_{2,I},& \\
&\norm{u-\Pi^L u}_{\infty,I} \le \frac{1}{8} \norm{u^{(2)}}_{\infty,I},&
\end{eqnarray*}

where $u^{(2)}$ denotes the second derivative of $u$, $\norm{\,\cdot\,}_{0,I}$ and $\norm{\,\cdot\,}_{\infty,I}$ denote the $L^2$- and $L^\infty$-norms, respectively, and $\abs{\,\cdot\,}_{1,I}$ and $\abs{\,\cdot\,}_{2,I}$ denote the $H^1$- and $H^2$-seminorms, respectively. The estimation presented above are optimal in the sense that there exist functions for which the equalities hold.
\begin{itemize}
\item Let $u(x) := \sin (\pi x)$ on the interval $(0,1)$. Then, $\Pi^L u(x) = 0$. In this case,
$$
\norm{u - \Pi^L u}_{0,I}  = \frac{1}{\pi^2}\abs{u}_{2,I}
,~~
\abs{u - \Pi^L u}_{1,I} = \frac{1}{\pi}\abs{u}_{2,I}
~.
$$

\item Let $u(x) := x^2$ on the interval $(0,1)$. Then, $\Pi^L u (x) = x$. In this case,
$$
\norm{u-\Pi^Lu}_{\infty,I} = \frac{1}{8} \norm{u^{(2)}}_{\infty,I}.
$$
\end{itemize}
\item {\bf (2D case)} Over a triangle $K$ with vertices $p_i$ ($i=1,2,3$), the Lagrange interpolation function $\Pi^L u$ is the linear function  such that (see Figure \ref{fig:interpolation})
$$
(u-\Pi^L u)(p_i)=0,\quad \forall i=1,2,3.
$$ 

\begin{figure}[htp]
\begin{center}
	\includegraphics[scale=1]{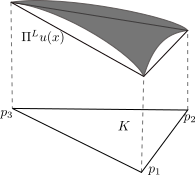}
\end{center}

\caption{\label{fig:interpolation} A linear Lagrange interpolation function $\Pi^L u$ defined on  a triangle $K$} 
\end{figure}

In case of $L^2$-norm and $H^1$-seminorm error estimation of $\Pi^L$,  one need to estimate the interpolation error constants appearing in the following inequalities:
$$
\norm{u-\Pi^L u}_{0,K} \le C_0(K) \abs{u}_{2,K},\quad
\abs{u-\Pi^L u}_{1,K} \le C_1(K) \abs{u}_{2,K}.
$$
Let $h$ be the medium edge length of $K$, $\theta$ the maximum angle, and $\alpha h$ ($0<\alpha\le 1$) the smallest edge length. Kikuchi and Liu \cite{kikuchi2007estimation, Liu-2010} obtained a bound of $C_0$ and $C_1$ as follows:
\begin{equation*}
C_0(K) \le \frac{h}{\pi}\sqrt{1+ \abs{\cos \theta}}, ~
C_1(K) \le 0.493h \frac{1+\alpha^2+\sqrt{1+2\alpha^2\cos2\theta + \alpha^4}}{\sqrt{2\left(1+\alpha^2-\sqrt{1+2\alpha^2\cos2\theta+\alpha^4}\right)}}.
\end{equation*}
Also, Kobayashi \cite{Kobayashi-2011} showed that for a triangle $K$ with edge lengths $A,B,C$ and area $S$, the following holds:
$$
\abs{u - \Pi^L u}_{1,K} \le C_1(K) \abs{u}_{2,K}, \quad \forall u \in H^2(K),
$$
where the constant $C_1(K)$ is defined by
$$
C_1(K) := \sqrt{\frac{A^2B^2C^2}{16S^2} - \frac{A^2 + B^2 + C^2}{30}-\frac{S^2}{5}\left( \frac{1}{A^2} + \frac{1}{B^2} + \frac{1}{C^2}\right)}.
$$

The optimal estimation of constants $C_0(K)$ and $C_1(K)$ for a concrete $K$ can be obtained by solving corresponding eigenvalue problems with rigorous lower eigenvalue bounds; see results of \cite{Liu-2018,Liao-2019}.

For $L^\infty$-norm error estimation under the $L^\infty$-norm of objective function, Waldron \cite{Waldron-1998} provides the following sharp inequality:
$$
\norm{u - \Pi^L u}_{\infty,K} \le \frac{1}{2}(R^2 - d^2) \norm{u^{(2)}}_{\infty,K},
$$

where $R$ is the radius of the circumscribed circle of $K$, $d$ is the distance of the center $c$ of the circumscribed circle from $K$, and $\norm{u^{(2)}}_{\infty,K}$ is defined by
$$
\norm{u^{(2)}}_{\infty,K} 
:= \sup_{x \in K}  \sup_{\substack{u,v \in \mathbb{R}^2\\ \norm{u}=\norm{v}=1}} \abs{D_{u} D_{v} u(x)} =
 \sup_{x \in K}  \sup_{\substack{\xi \in \mathbb{R}^2 \\ \norm{\xi}=1}} \abs{D^2_{\xi}u(x)}. 
$$
In particular, if $c \in K$,
$$
\norm{u - \Pi^L u}_{\infty, K} \le \frac{1}{2}R^2 \norm{u^{(2)}}_{\infty,K}.
$$
A detailed discussion on the $L^\infty$-norm of interpolation error for a quadratic polynomial $f$ is considered by D'Azevedo and Simpson \cite{D'Azevedo-1989}. In \cite{Shewchuk-2002}, Shewchuk gives a survey of the interpolation error estimation with $L^\infty$-norm for both $f-\Pi^L f$ and $\nabla(f-\Pi^L f)$, along with the discussion on the relation between the interpolation error and the finite element approximation error functions. Also, the discussion on the affection of the aspect ratio of triangle element to  the interpolation error can be found in Cao \cite{Cao-2005}. 

\end{itemize}

\medskip

In this research, we consider the $L^\infty$-norm estimation for the Lagrange interpolation over triangle element $K$ by using the $H^2$-seminorm of the objective function, that is, 
\begin{equation}
\label{eq:objective-problem}    
\norm{u-\Pi^{L} u}_{\infty,K} \le C^{L}(K) \abs{u}_{2,K},\quad \forall u \in H^2(K).
\end{equation}
Here $C^L(K)$ is the interpolation error constant to be evaluated explicitly. For example, for a unit isosceles right triangle element, we have an estimation of the optimal constant  $C^L(K)$ as 
$$
0.40432 \le  C^L(K) \le 0.41595 ~.
$$
Estimation of $C^L(K)$ for triangle of general shapes is provided in Theorem \ref{th-raw}, while sharp bounds for concrete triangles are discussed in \S \ref{sec3}.  
Such kind of estimation is helpful to provide explicit maximum norm error estimation for the FEM solution to boundary value problems by further applying the point-wise error estimation (see e.g.,  \cite{Fujita-1955}), which will be considered in our succeeding work; see also classical qualitative error analysis under the maximum norm in Section 19-22 of \cite{Ciarlet-1991};

\medskip

The contribution of our paper is summarized as follows.
\begin{itemize}
    \item [(1)] For triangle element $K$ of general shapes, a formula to give an upper bound of $C^L(K)$ is obtained by theoretical analysis. The bound is raw but works well for triangle element of arbitrary shapes. Particularly, our analysis tells that the value of $C^L(K)$ can be very large and tend to $\infty$ if the triangle element tends to degenerate to a 1D segment; see detail in \S2.2.
    \item [(2)]For specific triangle element $K$, the optimal estimation of $C^L(K)$ is obtained by solving the corresponding optimization problem over $H^2(K)$ under the constraint condition involving $L^\infty$-norm. The processing of the constraint condition with $L^\infty$-norm is not an easy work. We develop a novel algorithm to provide efficient and sharp estimation for the solution of the optimization problem. With a light computation, one can obtain the estimation of $C^L(K)$ with relative error less than 1\%. 
\end{itemize}

\medskip

The rest of our paper is structured as follows. At the end of this section, we introduce the preliminary concepts and notations to be used throughout the paper. In \S2, the estimation of the upper bound for $C^L(K)$ is considered using theoretical approach. The raw upper bound of the interpolation error constant is calculated for a right isosceles triangle. Also, we investigate the asymptotic behavior of the constant as the triangle tends to degenerate. In \S3, using finite element method (FEM), an algorithm for the optimal estimation of the constant is proposed. Lower bounds for the constant are calculated to confirm the efficiency of the proposed algorithm. The numerical results are summarized and the conclusion is presented in \S4.

\paragraph{Notation}Let us introduce the notation for the function spaces used in this paper. In most cases, the domain $\Omega$ of functions is selected as a triangle element $K$. The standard notation is used for Sobolev function spaces $W^{k,p}(\Omega)$. The associated norms and seminorms are denoted by $\norm{\,\cdot\,}_{k,p,\Omega}$ and $\abs{\,\cdot\,}_{k,p,\Omega}$, respectively (see, e.g., Chapter 1 of \cite{Brenner+Scott2008} and Chapter 1 of \cite{Ciarlet2002}). Particularly, for special $k$ and $p$, we use abbreviated notations as $H^k(\Omega)=W^{k,2}(\Omega)$, $\abs{\,\cdot\,}_{k,\Omega} = \abs{\,\cdot\,}_{k,2,\Omega}$, and $L^p(\Omega)=W^{0,p}(\Omega)$. The set of polynomials over $K$ of up to degree $k$ is denoted by $P_k(K)$. The second order derivative is given by $D^2 u := ( u_{xx}, u_{xy}, u_{yx}, u_{yy})$ for $u \in H^2(K)$.

Given a triangle $K$, denote each vertex by $p_i$ ($i=1,2,3)$ and the largest edge length by $h_K$; see Figure \ref{fig:tri}. We follow the notation introduced by Liu and Kikuchi \cite{Liu-2010} to configure a general triangle with geometric parameters. Let $h, \alpha$ and $\theta$ be positive constants such that
$$
h>0, \quad 0 < \alpha \le 1, \quad \left(\frac{\pi}{3} \le \right) \cos^{-1}\left(\frac{\alpha}{2}\right) \le \theta < \pi.
$$
Define a triangle $K_{\alpha,\theta,h}$ with three vertices $p_1(0,0)$, $p_2(h,0)$ and $p_3(\alpha h \cos \theta, \alpha h \sin \theta)$. Note that $h \le h_K$. In case of $h=1$, the notation $K_{\alpha,\theta,1}$ is abbreviated as $K_{\alpha,\theta}$.

\begin{figure}[htp]
\begin{center}
	\includegraphics[scale=0.35]{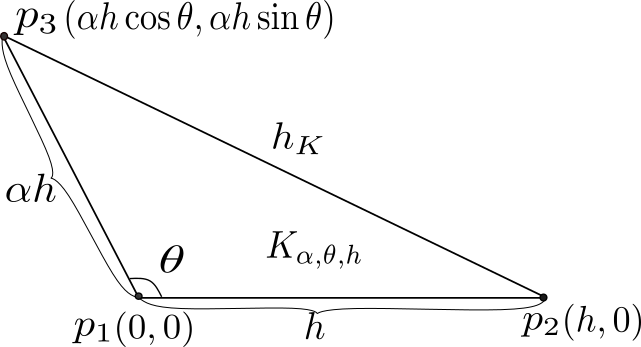}
\end{center}
\caption{\label{fig:tri} Configuration of triangle $K_{\alpha,\theta,h}$}
\end{figure}

With the above configuration of the triangle $K_{\alpha,\theta,h}$, the optimal constant $C^L(K)$ in \eqref{eq:objective-problem} can be defined as follows:
\begin{equation}
    \label{def:constant-CL-def}
C^L(\alpha,\theta,h) := \sup_{u\in H^2(K_{\alpha,\theta,h})} \frac{\abs{u}_{2,K_{\alpha,\theta,h}}}{ \norm{u-\Pi^L u}_{\infty,K_{\alpha,\theta,h}} }~.
\end{equation}
By scaling of the triangle element, it is easy to confirm that $C^L(\alpha,\theta,h)=h C^L(\alpha,\theta,1)$.

In the rest of the paper, we show how to obtain explicit bounds for the error constant $C^L(\alpha,\theta,h)$. 


\section{Raw upper bound of the constant}\label{sec2}
In this section, a raw upper bound of the constant is obtained through theoretical analysis. Such a bound applies to triangles of arbitrary shapes. 

First, let us quote a lemma about the trace theorem, which gives estimation for the integral over edge of a triangle element. For reader's convenience, we show the proof in a concise way; refer to e.g. \cite{vejchodsky2014robust, Carstensen-2014,Liu-2019} for more detailed discussion.

\begin{figure}[htp]
\begin{center}
	\includegraphics[scale=0.3]{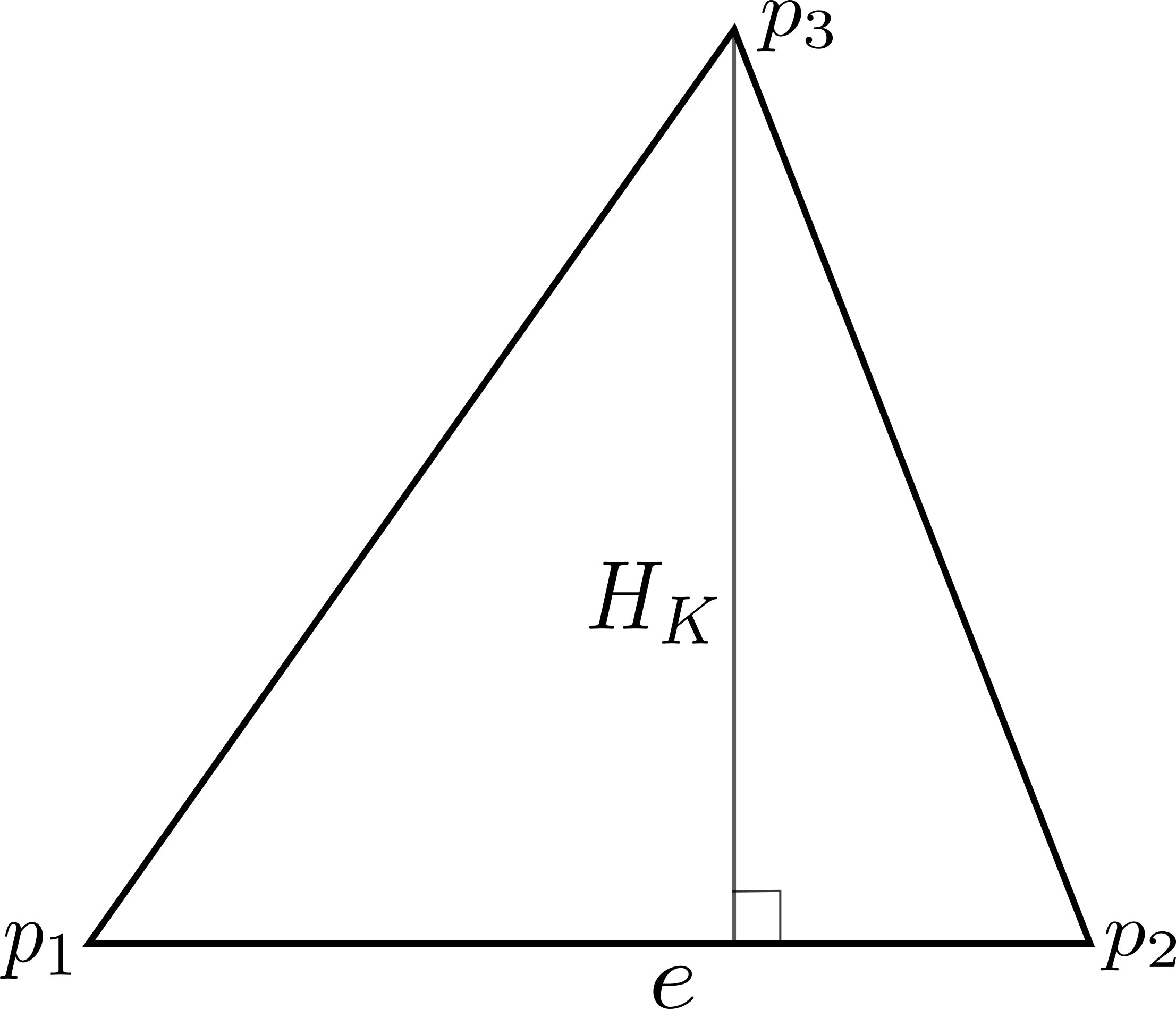}
\end{center}
\caption{\label{fig:trace} A triangle $K$ with base $e$ and height $H_K$}
\end{figure}

\begin{lem}[Trace theorem]\label{lem:trace_theorem}
Let $e$ be one of the edges of triangle $K$; see Figure \ref{fig:trace}. Given $w\in H^1({K})$, we have the following estimation:
$$
\norm{w}^2_{0,e} \le \frac{\abs{e}}{\abs{K}} \left\{\norm{w}_{0,K}^2 + h_K \norm{w}_{0,K} \abs{w}_{1,K} \right\}.
$$
\end{lem}
\begin{proof}
For any $w \in H^{1}(K)$, the Green theorem leads to
$$
\int_K ((x,y)- p_3 ) \cdot \nabla (w^2) \mbox{d}K = \int_{\partial K} ((x,y)- p_3 )\cdot {\overrightarrow{n}} w^2 \mbox{d}s - \int_K 2 w^2 \mbox{d} K.
$$
Here, $\overrightarrow{n}$ is the unit outer normal direction on the boundary of $K$. For the term $((x,y)- p_3 )\cdot {\overrightarrow{n}}$, we have
\begin{equation*}
((x,y)- p_3 )\cdot {\overrightarrow{n}} =
\left\{
\begin{array}{ll}
0   & \mbox{on }  p_1p_3, ~ p_2p_3\:, \\
H_K & \mbox{on } e\:.
\end{array}
\right.
\end{equation*}
Here, $H_K$ is the height of the triangle with base as $e$. Thus,

\begin{align*}
H_K \int_e w^2 \mbox{d} s & = \int_K 2 w^2 \mbox{d} K + \int_K ((x,y)- p_3 ) \cdot \nabla (w^2) \mbox{d} K\\
& \le \int_K 2 w^2 \mbox{d} K + 2h_K \int_K w \abs{\nabla w} \mbox{d} K \\
& \le 2\norm{w}_{0,K}^2 + 2h_K \norm{w}_{0,K} \norm{\nabla w}_{0,K}.
\end{align*}
We can now draw the conclusion by sorting the above inequality. 
\end{proof}

Using the trace theorem, the following result provides a pointwise estimation of the interpolation error.

\begin{lem} \label{lem-ptws}
Given $u\in H^2(K)$, for any point $\mathbf{x}_0 \in K$, we have
$$
\abs{(u-\Pi^L u)(\mathbf{x}_0)} \le \dfrac{\sqrt{2\abs{p_1\mathbf{x}_0}}}{\sqrt{H_{\widetilde{K}}}} \left(h_K \abs{u-\Pi^L u}_{1,K} \abs{u}_{2,K} + \abs{u-\Pi^L u}_{1,K}^2 \right)^{\frac{1}{2}},
$$
where $h_K$ is the longest edge length of $K$, and $H_{\widetilde{K}}$ is the height of the subtriangle $\widetilde{K} = p_1\mathbf{x}_0 p_3$ with respect to the base $\tilde{e} = p_1\mathbf{x}_0$ (see Figure \ref{fig:subtriKtilde}).
\end{lem}

\begin{figure}[htp]
\begin{center}
	\includegraphics[scale=0.3]{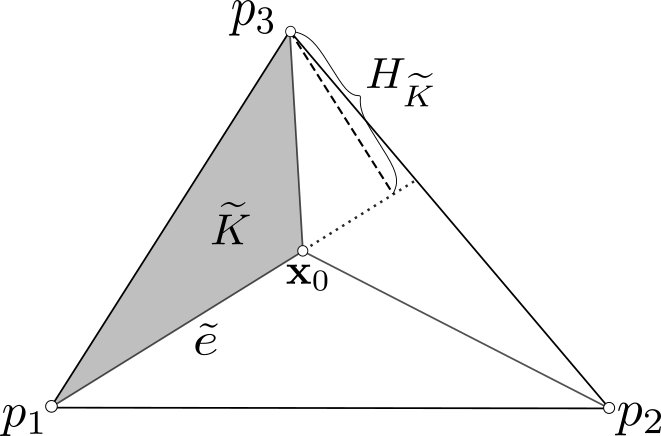}
\end{center}
\caption{\label{fig:subtriKtilde} A subtriangle $\widetilde{K}$ in a triangle $K$ }
\end{figure}

\begin{proof}
Let $g = u - \Pi^L u$ and $t$ be the direction along edge $p_1 \mathbf{x}_0$. In Lemma \ref{lem:trace_theorem}, by taking $\displaystyle w:=\frac{\partial g}{\partial t}$, we have 
$$
\norm{\frac{\partial g}{\partial t}}^2_{0,\tilde{e}} \le \frac{\abs{\tilde{e}}}{\lvert\widetilde{K}\rvert} \left(\norm{w}^2_{0,\widetilde{K}} + h_{\widetilde{K}} \norm{w}_{0,\widetilde{K}} \abs{w}_{1,\widetilde{K}} \right) \le \frac{\abs{\tilde{e}}}{\lvert\widetilde{K}\rvert} \left(\abs{g}^2_{1,\widetilde{K}} + h_{\widetilde{K}} \abs{g}_{1,\widetilde{K}} \abs{g}_{2,\widetilde{K}} \right).
$$
Taking the Taylor expansion of $g$ on the segment $\tilde{e}$ and noting that $g(p_1) = 0$, 
\begin{eqnarray*}
\abs{g(\mathbf{x}_0)} & = & \int_{p_1\mathbf{x}_0} \dfrac{\partial g}{\partial t}\mbox{d}t + g(p_1) 
\le \sqrt{\abs{p_1\mathbf{x}_0}} \cdot \norm{\frac{\partial g}{\partial t}}_{0,\tilde{e}} \notag\\
& \leq & \dfrac{\abs{p_1\mathbf{x}_0}}{\sqrt{|\widetilde{K}|}} \left(h_K \abs{g}_{1,\widetilde{K}} \abs{g}_{2,\widetilde{K}} + \abs{g}_{1,\widetilde{K}}^2 \right)^{\frac{1}{2}} \label{eq:in_proof_g_x0_est}
\\
& \le & \frac{\sqrt{2\abs{p_1\mathbf{x}_0}}}{\sqrt{H_{\widetilde{K}}}} \left(h_K \abs{g}_{1,K} \abs{g}_{2,K} + \abs{g}_{1,K}^2 \right)^{\frac{1}{2}}.
\end{eqnarray*}
The conclusion follows.
\end{proof}

Liu and Kikuchi \cite{Liu-2010} considered the estimation of the constant $C_1(\alpha,\theta)$ for different types of triangles $K = K_{\alpha,\theta}$ such that
$$
\abs{u - \Pi^L u}_{1,K}\le C_1(\alpha,\theta) h \abs{u}_{2,K}, \quad \forall u \in H^2(K),
$$
where $h$ is the medium length of $K$. The constant $C_1(\alpha,\theta)$ is used to give a bound for $C^L(K)$ as shown in the lemma below.

\begin{lem}\label{th-raw}
Given $u \in H^2(K)$, for any point $\mathbf{x}_0 \in K$, we have
\begin{equation}
\label{eq:est_l_infty}
\abs{(u -\Pi^L u)(\mathbf{x}_0)} \le \frac{\sqrt{2 \abs{p_1 \mathbf{x}_0}}}{\sqrt{H_{\widetilde{K}}}} \left(C_1(\alpha,\theta)hh_K + C^2_1(\alpha,\theta)h^2\right)^{\frac{1}{2}} \abs{u}_{2,K}.
\end{equation}
\end{lem}

\subsection{The case for a right isosceles triangle}
Using Lemma \ref{th-raw}, we obtain the upper bound of the constant for right isosceles triangles: For the right isosceles triangle $K=K_{1,\frac{\pi}{2},h}$,
\begin{equation}\label{eq:bound-rit}
\norm{u - \Pi^L u}_{\infty,K} \le 1.3712h \abs{u}_{2,K}.
\end{equation}

\begin{figure}[htp]
\begin{center}
	\includegraphics[scale=0.5]{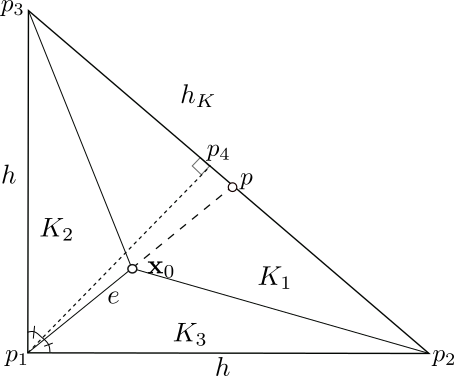}
\end{center}
\caption{\label{fig:righttri}A right isosceles triangle $K_{1,\frac{\pi}{2},h}$}
\end{figure}

Suppose a point $\mathbf{x}_0$ subdivides $K$ into $K_1, K_2, K_3$; see Figure \ref{fig:righttri}. 
Let us consider the estimation of the term $\abs{p_1\mathbf{x}_0}/H_{K_2}$, which is required in Lemma 
\ref{th-raw}. Let $p_1p_4$ be the height of $K$ with base as $p_2 p_3$. Due to the symmetry of $K$, it is enough to only consider the case that $\mathbf{x}_0 \in K$ is below the line $p_1p_4$.
Let $p$ be the intersection of the extended line of $p_1\mathbf{x}_0$ and edge $p_2p_3$. Note that $\abs{p_1\mathbf{x}_0} \le \abs{p_1p}$. 
For $p:=(\overline{x},\overline{y})$ on $p_2p_4$, $\abs{p_1p} = \sqrt{\overline{x}^2+\overline{y}^2}$. The height of $K_2$ with base $p_1\mathbf{x}_0$ is given by
$$
H_{K_2} = \frac{h\overline{x}}{\sqrt{\overline{x}^2+\overline{y}^2}}.
$$

Then, since $\overline{y} = h-\overline{x}$,
$$
\frac{\abs{p_1p}}{H_{K_2}} =  \frac{2\overline{x}^2-2h\overline{x}+h^2}{h\overline{x}}.
$$

The above quantity takes its maximum value at $p = \left(\frac{h}{2},\frac{h}{2}\right)$ and $p = (h,0)$, and its maximum value is $1$. 
Thus, for any $p$ on $p_2p_4$, $\abs{p_1p}/H_{K_2} \le 1$. From \cite{Liu-2010}, $C_1\left(1,\frac{\pi}{2}\right) \le 0.49293$. Since $h_K = \sqrt{2}h$, by inequality (\ref{eq:est_l_infty}), 
$$
\abs{(u -\Pi^L u)(\mathbf{x}_0)} \le \sqrt{2} \left[(0.49293)\sqrt{2} + 0.49293^2\right]^{\frac{1}{2}} h \abs{u}_{2,K} \le 1.3712h \abs{u}_{2,K}.
$$
Hence, we obtain the error estimate for a right isosceles triangle as in \eqref{eq:bound-rit}.

\subsection{Dependence of the constant on the shape of $K$}

In this subsection, we consider the variation of the interpolation constant when a reference triangle, i.e., the right isosceles triangle, is transformed to a general triangle.
\begin{thm}
\label{thm:formula-of-general-bound}
For a general element $K_{\alpha,\theta}$, the following estimation for constant $C^L(\alpha, \theta)$ holds:
$$
C^L(\alpha, \theta) \le \frac{v_+(\alpha, \theta)}{2\sqrt{\alpha \sin \theta}} C^L\left(1,\frac{\pi}{2}\right)~,
$$ 
where $v_+(\alpha, \theta) = 1 + \alpha^2 + \sqrt{1+2\alpha^2\cos 2\theta + \alpha^4}$.
\end{thm}

\begin{proof}
Let us consider the affine transformation between $x = (x_1,x_2) \in K_{\alpha,\theta}$ and $\xi = (\xi_1,\xi_2) \in K_{1,\frac{\pi}{2}}$:
$$
\xi_1 = x_1 - \frac{x_2}{\tan \theta}, \quad \xi_2 = \frac{x_2}{\alpha \sin \theta}, \quad\mbox{or}\quad x_1 = \xi_1 + \alpha \xi_2 \cos \theta, \quad x_2 = \alpha \xi_2 \sin \theta.
$$

Given $\tilde{v}(\xi)$ over $K_{1,\frac{\pi}{2}}$, define $v(x)$ over $K_{\alpha,\theta}$ by $v(x_1,x_2) = \tilde{v}(\xi_1, \xi_2)$. Thus,
$$
\norm{v}_{\infty, K_{\alpha,\theta}} =
\norm{\tilde{v}}_{\infty, K_{1,\frac{\pi}{2}}}~. 
$$
The estimation for the variation of $H^2$-seminorm in Theorem 1 of \cite{Liu-2010} tells that
$$
\abs{v}_{2, K_{\alpha,\theta}} \ge
\frac{2 \sqrt{\alpha \sin \theta}}{v_+(\alpha,\theta)} 
\abs{\tilde{v}}_{2,K_{1,\frac{\pi}{2}}} ~.
$$
Thus, we draw the conclusion from the definition of constant $C^L(\alpha,\theta)$ in (\ref{def:constant-CL-def}).
\end{proof}
\begin{remark}
By using the raw bound of $C^L\left(1,\frac{\pi}{2}\right) \le 1.3712 h $ in (\ref{eq:bound-rit}), an explicit but raw bound of $C^L(\alpha,\theta)$ is available.
Later, with a sharp and rigorous estimation of $C^L\left(1,\frac{\pi}{2}\right)$ based on numerical approach, the bound can be improved as 
\begin{equation}
\label{eq:estimation-of-CL-general-shape} 
C^L(\alpha, \theta,h) 
\le 0.41595 h \frac{v_+(\alpha, \theta)}{2\sqrt{\alpha \sin \theta}}~.
\end{equation}
\end{remark}

\begin{remark}
Here are two remarks on the asymptotic behavior of the constant when the triangle tends to degenerate.
\begin{enumerate}
\item Suppose the maximum inner angle $\theta$ of $K_{\alpha,\theta}$ is close to $\pi$;
see Figure \ref{fig:obttri}. 
Let $u(x,y) := x^2 + y^2$. Then, $\Pi^L u (x,y) = x + ((\alpha - \cos \theta)/\sin \theta) y$ and
$$
\norm{u - \Pi^L u}_{\infty,K_{\alpha,\theta}} = (2\alpha \cos \theta - \alpha^2 - 1)/4, \quad 
\abs{u}_{2,K_{\alpha,\theta}} = 2 \sqrt{\alpha \sin \theta}.
$$
Thus, we have a lower bound of $C^L(\alpha,\theta)$ as follows, 
$$
C^L(\alpha,\theta) \ge \frac{2\alpha \cos \theta - \alpha^2 -1}{8\sqrt{\alpha \sin \theta}}~.
$$
In this case, $C^L(\alpha,\theta)$ diverges to $\infty$ as $\theta$ tends to $\pi$.
\begin{figure}[htp]
\begin{center}
	\includegraphics[scale=0.3]{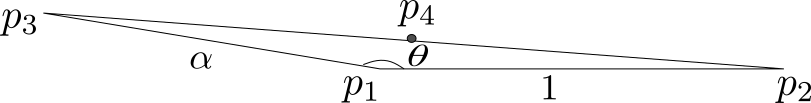}
\end{center}
\caption{\label{fig:obttri}A triangle $K_{\alpha,\theta}$ with angle $\theta$ close to $\pi$}
\end{figure}

\item For triangle $K_{\alpha, \frac{\pi}{2}}$ shown in Figure \ref{fig:smrighttri}, let $u(x,y) := \abs{(x,y)- p_4}^2$, where $p_4$ is the midpoint of the edge $p_2p_3$. Then, $\Pi^L u = (\alpha^2 + 1)/4$ and 
$$
\norm{u - \Pi^L u}_{\infty,K_{\alpha, \frac{\pi}{2}}} = (\alpha^2 +1)/4, \quad 
\abs{u}_{2,K_{\alpha, \frac{\pi}{2}}} = 2\sqrt{\alpha}~.
$$
Thus,
$$
\frac{\norm{u - \Pi^L u}_{\infty,K_{\alpha, \frac{\pi}{2}}}}{\abs{u}_{2,K_{\alpha, \frac{\pi}{2}}}} = \frac{\alpha^2+1}{8\sqrt{\alpha}} ~
\left(\le C^L\left(\alpha,\frac{\pi}{2}\right) \right).
$$
In  case that $\alpha \to 0$, although the maximum inner angle is invariant, 
the interpolation error constant $C^L\left(\alpha,\frac{\pi}{2}\right)$ tends to $\infty$ .

\begin{figure}[htp]
\begin{center}
	\includegraphics[scale=0.35]{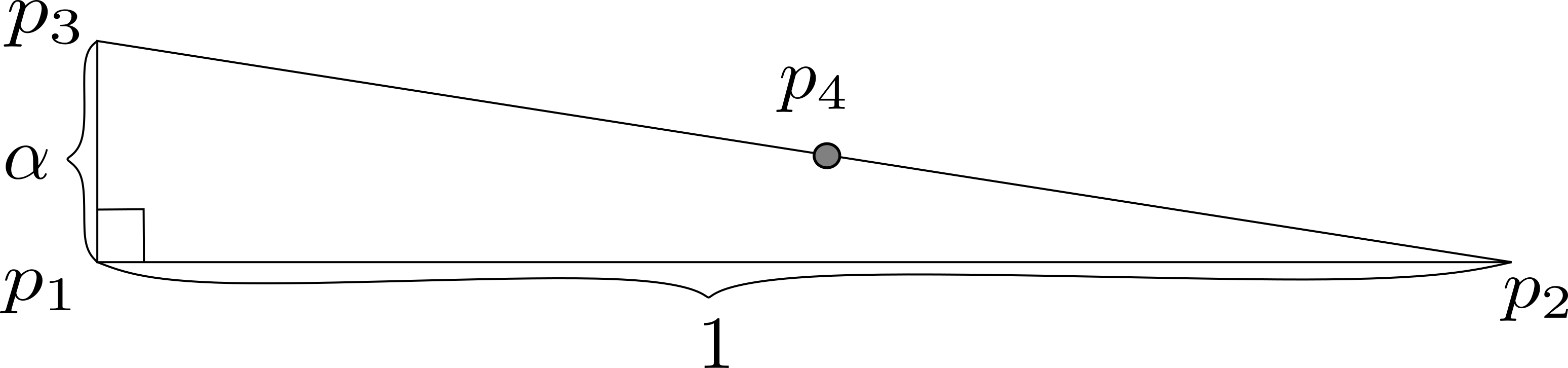}
\end{center}
\caption{\label{fig:smrighttri}A right triangle $K_{\alpha, \frac{\pi}{2}}$ with one leg length close to $0$. }
\end{figure}
\end{enumerate}
\end{remark}

\section{Optimal estimation of the constant}\label{sec3}
In the previous section, we obtain explicit bounds for the interpolation constant for triangles of general shape. Basically, such bounds from theoretical analysis only provide raw bound for the objective constant. In this section, we propose a numerical algorithm to obtain the optimal estimation of the constant $C^L(K)$ for specific triangles.

Let us define the space $V^L(K) := \left\{u \in H^2(K) \mid u(p_i) = 0 \ (i = 1,2,3) \right\}$. 
Let $\mathcal{T}^h$ be a triangulation of the domain $K$ and define the space
\begin{multline*}
V^{\FM}_h(K) := \biggl\{ u_h \biggm| u_h|_{T} \in P_2(T), \ \forall T \in \mathcal{T}^h;\  u_h(p_i) = 0\ (i = 1,2,3); \ u_h \mbox{ is continu-}\biggr.\\
\biggl.\mbox{ous on each vertex of } \mathcal{T}^h; \ \int_{e}\frac{\partial u_h}{\partial n} \mbox{d}s \mbox{ is continuous on each interior edge } e \mbox{ of } \mathcal{T}^h \biggr\}~.
\end{multline*}
For $u_h, v_h \in V^{\FM}_h(K)$, define the discretized $H^2$-inner product and seminorm by 
$$
\langle u_h, v_h \rangle_h := \sum_{T \in \mathcal{T}^h} \int_T D^2u_h \cdot D^2 v_h ~\text{d}T,\quad \abs{u_h}_{2,K} := \sqrt{\langle u_h, u_h \rangle_h}~.
$$
Let us define the two quantities over the triangle $K$:
\begin{equation}
    \label{eq:def-lambda-lambda_h}
\lambda(K) := \inf_{u \in V^{L}(K)} \frac{\abs{u}_{2,K}^2}{\norm{u}_{\infty,K}^2}, \qquad \lambda_{h}(K) := \min_{u_h \in V_h^{\FM}(K)} \frac{\abs{u_h}_{2,K}^2}{\norm{u_h}_{\infty,K}^2}.
\end{equation}
Note that $C^L(K) = \sqrt{\lambda(K)}^{-1}$ holds. In Theorem \ref{th-main}, we describe the algorithm to bound $\lambda$ by using $\lambda_h$.

\vspace{0.2in}

Given $u \in H^2(K)$, the Fujino--Morley interpolation $\Pi^{\FM}_h u$ is a function satisfying 
$$
\Pi^{\FM}_h \in V^{\FM}_h(K); \quad \Pi^{\FM}_h u |_T \in P_2(T), \quad \forall T \in \mathcal{T}^h,
$$
and at the vertices $p_i$ and edges $e_i$ of $K$, 
$$
(u - \Pi^{\FM}_h u)(p_i) = 0, \quad \int_{e_i} \frac{\partial}{\partial n}(u - \Pi^{\FM}_h u) \mbox{d}s = 0 \quad (i = 1,2,3)~.
$$
The Fujino--Morley interpolation has the property that (see, e.g., \cite{Liu-2018,Liao-2019}) 
\begin{equation}\label{eq:FM-property}
\langle u - \Pi^{\FM}_h u, v_h \rangle_h = 0, \quad \forall v_h \in V^{\FM}_h(K).
\end{equation}
Let $V(h) := \left\{u + u_h \mid u \in V^{L}(K), \ u_h \in V^{\FM}_h(K)\right\}$. Thus, it is easy to see that the Fujino--Morley interpolation is just the projection $P_h: V(h)\to V^{\FM}_h(K)$ with respect to the inner product $\langle \cdot, \cdot \rangle_h$. 

Below, let us introduce the theorem that provides an explicit lower bound of $\lambda$. 
Such a result is inspired by idea of \cite{LIU2015341} for the lower bounds of eigenvalue problems. 
\begin{thm}\label{th-main}
Suppose there exists a quantity $C_h^{\FM}$ such that, 
\begin{equation}
\label{eq:CFM_prop}
\norm{u - \Pi^{\FM}_h u}_{\infty,K} \le C^{\FM}_h \abs{u - \Pi^{\FM}_h u}_{2,K},\quad 
\forall u \in V^L(K)
\:.
\end{equation}
Then, the following lower found of $\lambda(K)$ is available.
\begin{equation}
\label{eq:lambda_CFM}
\lambda(K) \ge \frac{\lambda_{h}}{1+(C_h^{\FM})^2\lambda_{h}}~. 
\end{equation}
\end{thm}

\begin{proof}
For any $u \in V^L(K)$, noting that $\abs{\Pi^{\FM}_h u}_{2,K} \ge \sqrt{\lambda_h}\norm{\Pi^{\FM}_h u}_{\infty,K}$ and applying the inequality (\ref{eq:CFM_prop}), we have
\begin{align*}
\norm{u}_{\infty,K}& = \norm{\Pi_h^{\FM}u + u-\Pi_h^{\FM}u}_{\infty,K}\\
& \le \norm{\Pi_h^{\FM}u}_{\infty,K} + \norm{u-\Pi_h^{\FM}u}_{\infty,K}\\
& \le \frac{\abs{\Pi_h^{\FM}u}_{2,K}}{\sqrt{\lambda_h}} + C_h^{\FM}\abs{u-\Pi_h^{\FM}u}_{2,K}\\
& \le \sqrt{\frac{1}{\lambda_h} + (C_h^{\FM})^2} \sqrt{\abs{\Pi_h^{\FM}u}^2_{2,K} + \abs{u-\Pi_h^{\FM}u}^2_{2,K}}~.
\end{align*}
From the orthogonality in \eqref{eq:FM-property}, we have
$$
\abs{\Pi_h^{\FM}u}^2_{2,K} + \abs{u-\Pi_h^{\FM}u}^2_{2,K} = \abs{u}_{2,K}^2.
$$
Thus,
$$
\norm{u}_{\infty,K} \le \sqrt{\frac{1+(C_h^{\FM})^2\lambda_h}{\lambda_h}}\abs{u}_{2,K},~~ \forall u \in V^L(K)\:.
$$
From the definition of $\lambda$ in \eqref{eq:def-lambda-lambda_h}, we draw the conclusion.
\end{proof}
 
To apply Theorem~\ref{th-main} for bounding $\lambda$, an explicit value of $C_h^{\FM}$ is needed. Below, let us describe the way to obtain this explicit value by utilizing the raw bound of $C^L(\alpha,\theta)$.

\subsection{Explicit estimation of $C^{\FM}_h$}

To have an explicit value of $C^{\FM}_h$, we first define the quantity $C^{\FM}_{res}(T)$ for each element $T$ in the triangulation $\mathcal{T}^h$:

$$
C^{\FM}_{res}(T) := \sup_{u \in H^2(T)} \frac{\norm{u-\Pi_h^{\FM}u}_{\infty,T}}{\abs{u-\Pi_h^{\FM}u}_{2,T}} = \sup_{w \in W_1} \frac{\norm{w}_{\infty,T}}{\abs{w}_{2,T}}.
$$
Here, $W_1 := \left\{w \in H^2(T) \mid w(p_i) = 0, \ \int_{e_i} \frac{\partial w}{\partial n }\mbox{d}s = 0 \ (i = 1,2,3) \right\}$. 
Noticing that  $W_1 \subseteq W_2$ for $W_2 := \left\{ w \in H^2(T) \mid w(p_i) = 0 \ (i=1,2,3) \right\}$, from the definition of $C^L$ in (\ref{def:constant-CL-def}), we have
$$
C^{\FM}_{res}(T) \le \sup_{w \in W_2} \frac{\norm{w}_{\infty,T}}{\abs{w}_{2,T}} = C^{L}(T).
$$
Then, the following $C^{\FM}_h$ with an upper bound makes certain (\ref{eq:CFM_prop}) holds:
\begin{equation}
\label{eq:def-Ch-FM}
    C^{\FM}_h := \max_{T \in \mathcal{T}^h} C^{\FM}_{res}(T)
    \left(
     \le 
    \max_{T \in \mathcal{T}^h} C^L(T) \right) ~. 
\end{equation}

\begin{figure}[htp]
\begin{center}
	\includegraphics[scale=0.4]{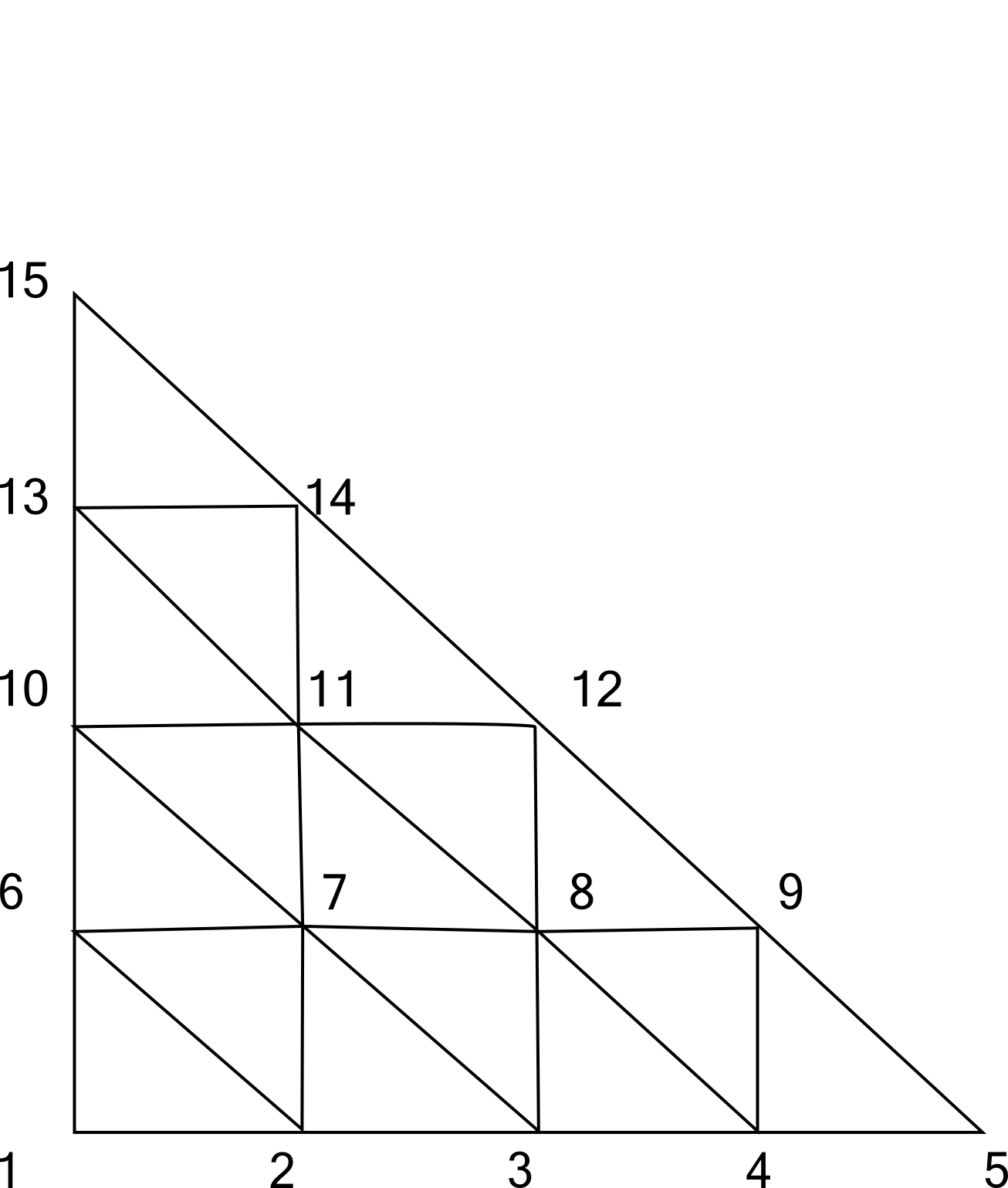}
\end{center}
\caption{A uniform triangulation of a right isosceles triangle}
\label{fig:unifmesh}
\end{figure}
\begin{remark}
Let $\mathcal{T}^h$ be a uniform triangulation of a right isosceles triangle; see a sample mesh in Figure \ref{fig:unifmesh}. We choose an explicit upper bound of $C^{\FM}_h $ as $C^{\FM}_h \le 1.3712h$, since for each $T \in \mathcal{T}^h$, $C^{\FM}_{res} \le C^L(T) \le 1.3712h$, where $h$ is the leg length of each right triangle element.
\end{remark}

\subsection{Estimation of $\lambda_h$ by solving finite dimensional optimization problem}
In this part, we present a method to estimate $\lambda_h$, which is required in Theorem \ref{th-main} for bounding $\lambda$. Let $M := \mbox{Dim}(V^{\FM}_h)$. The estimation of $\lambda_h$ is equivalent to finding the solution to the optimization problem
\begin{equation}
\label{eq:optimprob}
\lambda_h =\mbox{min}\quad \mathbf{x}^T \mathbf{A} \mathbf{x}, \quad \mbox{subject to}\quad \norm{\sum_{i=1}^M \mathbf{x}_i \phi_i}_{\infty,K} \ge 1 ~,
\end{equation}
where the components $a_{ij}$ of $\mathbf{A}$ are given by $a_{ij} = \langle \phi_i, \phi_j \rangle_h$, $\set{\phi_i}_{i=1,\ldots,M}$ are the basis functions for the Fujino--Morley space $V^{\FM}_h$, and $\mathbf{x}\in\mathbb{R}^M$ denotes the Fujino--Morley coefficient vector of $u_h \in V^{\FM}_h$.

To solve the optimization problem \eqref{eq:optimprob} is not an easy work since the $L^\infty$-norm of the function appears in the constraint. Here, we introduce the technique to apply Bernstein polynomials and their convex-hull property to solve the problem. Strictly speaking, a new optimization problem \eqref{eq:optimprob2} utilizing the Bernstein polynomials will be formulated to provide a lower bound for the solution of \eqref{eq:optimprob}.

\medskip 

As a preparation, let us introduce the definition of Bernstein polynomials along with the convex-hull property; refer to, e.g., \cite{Chang-1984, Farouki-2012} for detailed discussion.

\paragraph{Convex-hull property of Bernstein polynomials}
Given a triangle $K$, let $(u,v,w)$ be barycentric coordinates for a point $x$ in $K$. A Bernstein polynomial $u_h$ of degree $n$ over a triangle $K$ is defined by 
$$
u_h := \sum_{i+j+k = n} d_{i,j,k} J^{(n)}_{i,j,k}, \quad J^{(n)}_{i,j,k}(x) := \frac{n!}{i!j!k!} u^i v^j w^k~.
$$
Here, $J^{(n)}_{i,j,k}(x)$ are the Bernstein basis polynomials; the coefficients $d_{i,j,k}$ are the control points of $u_h$. Noticing that 
$$
J^{(n)}_{i,j,k}\ge 0,\quad 
\sum_{i+j+k = n} J^{(n)}_{i,j,k}=1,
$$
we can easily obtain the following convex-hull property of Bernstein polynomials:
$$
\norm{u_h}_{\infty,K} \le \max \abs{d_{i,j,k}}~.
$$

Given $u_h \in V^{\FM}_h(K)$, for each $T \in \mathcal{T}^h$, $u_h|_T \in P_2(T)$ can be represented by the Bernstein basis polynomials. Let $\mathbf{B}$ be the $N \times M$ matrix that transforms the Fujino--Morley coefficients $\mathbf{x}$ to the Bernstein coefficients $d^B$. Note that $u_h$ is regarded as a piecewise Bernstein polynomial so that its Bernstein coefficient vector $d^B$ has the dimension $N=6\times \#
\{elements\}$. From the convex-hull property of the Bernstein polynomials, the following inequality holds:
%
%
$$
1 \le \norm{\sum_{i=1}^M \mathbf{x}_i \phi_i}_{\infty,K} \le \norm{\mathbf{Bx}}_{\infty}. 
$$
Based on this inequality, we propose a new optimization by relaxing the constraint condition of \eqref{eq:optimprob}:
\begin{equation}
\label{eq:optimprob2}
\lambda_{h,B} = \mbox{min}~\mathbf{x}^T \mathbf{Ax}, \quad
\mbox{subject to}\quad  \norm{\mathbf{Bx}}_{\infty} \ge 1 ~.
\end{equation}
The solution to problem (\ref{eq:optimprob2}) provides a lower bound for \eqref{eq:optimprob}, i.e., $\lambda_{h} \ge \lambda_{h,B}$.

Below, we propose an algorithm to solve the problem (\ref{eq:optimprob2}). Since $\mathbf{A}$ is positive definite, let us consider the Cholesky decomposition of $\mathbf{A}$: $\mathbf{A} = \mathbf{R}^T\mathbf{R}$, where $\mathbf{R}$ is an $M\times M$ upper triangular matrix. Then, by letting $\mathbf{y} := \mathbf{Rx}$ and $\mathbf{\widehat{B}} := \mathbf{BR}^{-1}$, problem \eqref{eq:optimprob2} becomes
\begin{equation}
\label{eq:optimprob3}
\lambda_{h,B}=\mbox{min}~ \mathbf{y}^T \mathbf{y}, 
\quad
\mbox{subject to}\quad  \norm{\widehat{\mathbf{B}} \mathbf{y}}_{\infty} \ge 1~.
\end{equation}

The following lemma shows the solution for problem \eqref{eq:optimprob3}.
\begin{lem}\label{lem:optimsoln}\footnote{Appreciation to Tamaki TANAKA and Syuuji YAMADA from Faculty of Science, Niigata University for their idea of solving this problem in an efficient way.}
Let ${b^T_i}$ ($i=1,\ldots,N$) be the $i$th row of $\widehat{\mathbf{B}}$ and $b^T_{max}$ be a row of $\widehat{\mathbf{B}}$ satisfying $\norm{b_{max}}_2 = \max_{i=1,\ldots,N}\norm{b_i}_2$. Then, the optimal value of problem (\ref{eq:optimprob3}) is given by
$$
\lambda_{h,B} =\frac{1}{\norm{b_{max}}_2^2}.
$$
\end{lem}

\begin{proof}
Let $S := \left\{\mathbf{y} \mid \norm{\widehat{\mathbf{B}}\mathbf{y}}_\infty \ge 1\right\}$ and $\bar{\mathbf{y}} := \norm{b_{max}}_2^{-2}b_{max}$. Then, we have $\bar{\mathbf{y}}\in S$ because
$$
\norm{\widehat{\mathbf{B}}\bar{\mathbf{y}}}_\infty = \max_{i=1,\ldots,N} \abs{{b^T_i}\bar{\mathbf{y}}} \ge \abs{{b^T_{max}}\bar{\mathbf{y}}} = 1.
$$
Hence, 
\begin{equation}
\label{eq:lemproof1}
\min_{\mathbf{y}\in S} \mathbf{y}^T\mathbf{y} \le \bar{\mathbf{y}}^T\bar{\mathbf{y}} = \frac{1}{\norm{b_{max}}_2^{2}}.
\end{equation}
For any $\mathbf{y} \in S$, from the Cauchy--Schwarz inequality,
$$
1 \le \max_{i=1,\ldots,N} \abs{{b_i}^T\mathbf{y}} \le \max_{i=1,\ldots,N}\norm{b_i}_2 \norm{\mathbf{y}}_2 = \norm{b_{max}}_2 \norm{\mathbf{y}}_2.
$$
Thus,
\begin{equation}
\label{eq:lemproof2}
\frac{1}{\norm{b_{max}}_2^2} \le \min_{\mathbf{y}\in S} \mathbf{y}^T\mathbf{y}.
\end{equation}
From \eqref{eq:lemproof1} and \eqref{eq:lemproof2}, we draw the conclusion.
\end{proof}

Note that the diagonal elements of $\mathbf{B}\mathbf{A}^{-1}\mathbf{B}^T=\widehat{\mathbf{B}}\widehat{\mathbf{B}}^T$ correspond to $\|b_i\|_2^2$ ($i=1,\cdots, N$). Therefore, we can solve problem (\ref{eq:optimprob2}) without performing the Cholesky decomposition of $\mathbf{A}$, as shown by the following lemma.

\begin{lem}\label{lem:opti_problem_lower_bound}
Let $\mathbf{D}:=\mathbf{B} \mathbf{A}^{-1}\mathbf{B}^T$.
The optimal value of  (\ref{eq:optimprob2}) is given by
$$
\lambda_{h,B}=\frac{1}{\max(\mathrm{diag}(\mathbf{D}))},
$$
where $\mathrm{diag}(\mathbf{D})$ is the diagonal elements of $\mathbf{D}$.
\end{lem}

Theorem~\ref{th-main} gives a lower bound for $\lambda$. Since $C^L(K) = \sqrt{\lambda(K)}^{-1}$, this lower bound is used to obtain an upper bound for $C^L(K)$. 
Below, let us summarize the procedure to obtain a lower bound for $\lambda$.

\medskip

{\bf Algorithm for calculating lower bound of $\lambda(K)$}

\begin{enumerate}
    \item[a.] Set up the FEM space $V^{\FM}_h(K)=\mbox{span}\{\phi_i\}_{i=1}^M$ over a triangulation of the triangle domain $K$.
    \item[b.] Assemble the global matrix $\mathbf{A} = \left ( a_{ij} \right)_{M \times M}$ ($a_{ij} = \langle \phi_i, \phi_j \rangle_h$) and the transformation matrix $\mathbf{B}$ from Fujino--Morley coefficients to Bernstein coefficients.
    \item[c.] Apply Lemma~\ref{th-raw} to obtain a raw bound for $C^{\FM}_h$.
    \item[d.] Apply Lemma~\ref{lem:optimsoln} or Lemma~\ref{lem:opti_problem_lower_bound} to calculate $\lambda_{h,B} (\le \lambda_h)$.
    \item[e.] The lower bound for $\lambda$ is obtained through Theorem~\ref{th-main} by using $\lambda_{h,B}$ and the upper bound of $C^{\FM}_h$.
\end{enumerate}

\medskip

Using uniform triangulation of a domain $K$, a direct estimation of the lower bound for $\lambda$ without using $C^{\FM}_h$ is available.

\begin{cor}\label{th-cor}
For a uniform triangulation of $K=K_{\alpha,\theta,h}$ with $N$ subdivisions for each side, the following holds:
\begin{equation}
    \label{eq:lambda_lower_bound_uniform_mesh}
    \lambda(K) \ge \lambda_h(1-(1/N)^2).
\end{equation}
\end{cor}

\begin{proof}
Since $(C^L(K))^2 = 1/\lambda(K)$ and each $T \in \mathcal{T}^h$ is similar to $K$, we have, 
$$
\lambda (K) \ge \frac{\lambda_h}{1+ (C_h^{\FM})^2  \lambda_h } \ge \frac{\lambda_h}{1+(C^L(T))^2 \lambda_h } = \frac{\lambda_h}{ 1 + (1/N)^2 \lambda_h/\lambda(K)}.
$$
The conclusion is achieved by sorting the inequality.
\end{proof}

\begin{remark}
Theoretically, for a refined uniform triangulation, the lower bound (\ref{eq:lambda_CFM}) using $C_h^{\FM}$ is sharper (i.e., larger) than (\ref{eq:lambda_lower_bound_uniform_mesh}). This fact can be confirmed by utilizing the following relation:
\begin{align}
\label{eq:lower-bound-comparison}
\frac{\lambda_h}{1+ (C_h^{\FM})^2  \lambda_h }  \ge 
\lambda_h(1-(1/N)^2)
\iff
  1 \ge (N^2-1)  (C_h^{\FM})^2 \lambda_h ~.
\end{align}
For a small value of $h=1/N$, we have
$$
(N^2-1) (C_h^{\FM})^2 \approx  (N C_h^{\FM})^2 = (C_{res}^{\FM}(T))^2, ~ \lambda_h\approx \lambda=(C^L(T))^{-2}.
$$
Thus, the second equality of (\ref{eq:lower-bound-comparison}) holds due to $C_{res}^{\FM}(T) < C^L(T)$.
However, in practical computation, the raw estimate of $C_{res}^{\FM}(T)$ will cause a worse bound of $\lambda$ than (\ref{eq:lambda_lower_bound_uniform_mesh}).
\end{remark}

Using Corollary \ref{th-cor}, the following steps are modified from the algorithm to obtain a lower bound for $\lambda$,  without using the quantity of $C_h^{\FM}$:

\medskip

{\bf Revision of algorithm for calculating lower bound of $\lambda(K)$}

\begin{enumerate}
    \item[c*.] Apply Lemma~\ref{lem:optimsoln} or Lemma~\ref{lem:opti_problem_lower_bound} to calculate $\lambda_{h,B} (\le \lambda_h)$.
    \item[d*.] Solve the lower bound for $\lambda$ using Corollary \ref{th-cor} along with $\lambda_{h,B}$.
\end{enumerate}

\begin{remark}
To compare the efficiencies of the two formulas \eqref{eq:lambda_CFM} and \eqref{eq:lambda_lower_bound_uniform_mesh}, we apply them to estimate $\lambda$ for a unit right isosceles $K_{1,\pi/2}$. By using uniform triangulation of size $h=1/64$, the estimate \eqref{eq:lambda_CFM} gives $\lambda \ge 5.7659$ and \eqref{eq:lambda_lower_bound_uniform_mesh} gives a sharper bound as $\lambda \ge 5.7798$.
Hence, a sharper upper bound is obtained using \eqref{eq:lambda_lower_bound_uniform_mesh} and we have the following estimation:
$$
\norm{u-\Pi^L u}_{\infty,K_{1,\pi/2,h}} \le 0.41595 h \abs{u}_{2,K_{1,\pi/2,h}}.
$$
As a comparison, the  result \eqref{eq:bound-rit} will yield
a raw bound as $C^L(1,\pi/2,h) \le 1.3712h$. 
\end{remark}



For a triangle $K_{\alpha,\theta}$ with two fixed vertices $p_1(0,0), p_2(1,0)$, let us vary the vertex $p_3(x, y)$ and calculate the approximate value of $C^L(\alpha,\theta)$ for each position of $p_3$. Note that $C^L$ can be regarded as a function with respect to the coordinate $(x,y)$ of $p_3$, which is denoted by $C^L(x,y)$. In Figure \ref{fig:contour}, we draw the contour lines of $C^L(x,y)$ where the abscissa and the ordinate denote $x$- and $y$- coordinates of $p_3$, respectively.

\begin{figure}[htp]
\begin{center}
	\includegraphics[scale=0.5]{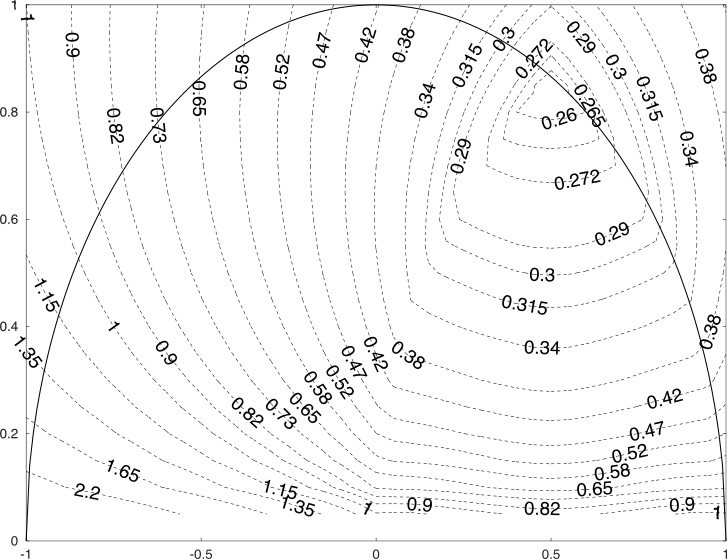}
\end{center}
\caption{\label{fig:contour}Contour lines of $C^L(\alpha,\theta)$ w.r.t. vertex $p_3(x,y)$}
\end{figure}

\subsection{Lower bound of the constant}\label{subsec3}

To confirm the precision of the obtained estimation for the Lagrange interpolation constant, the lower bounds of the constants are calculated. Let $u_h$ be the function obtained by numerical computation solving the minimization problem.  To obtain the lower bound, an appropriate polynomial $f$ over $K$ of higher degree $d$ is selected by solving the minimization problem below:
$$
\min_{f \in P_d(K)} \sum_{i=1}^n \abs{f(p_i) - u_h(p_i)}^2 \quad (n: \# \{\mbox{nodes of triangulation}\} )
$$
where $p_i$ denote the nodes of the triangulation of $K$. 
From the definition of $\lambda(K)$ in (\ref{eq:def-lambda-lambda_h}) and the relation  $C^L(K)=1/\sqrt{\lambda(K)}$, we have a lower bound of $C^L(K)$ as follows:
$$
C^L(K) \ge \frac{\norm{f}_{\infty,K}}{\abs{f}_{2,K}}~.
$$
\begin{remark}
For the unit right isosceles triangle $K_{1, \pi/2}$, the upper bound for the constant is obtained by solving the optimization problem with mesh size $1/64$. Meanwhile, the lower bound of the constant is obtained by using a polynomial of degree $9$. The two-side bounds reads:
$$
 0.40432 \le C^L\left(1,\frac{\pi}{2}\right) \le 0.41595.
$$
\end{remark}

\section{Numerical results and conclusion}\label{sec4}
In this section, we 
perform numerical computation to obtain the estimation of the interpolation error constant $C^L(K)$ for triangles of various shapes.

First, let us confirm the shape of the function $u_h$ that solves the minimization problem for $\lambda_{h,B}$ in case $K$ being the unit isosceles right triangle. The contour lines of $u_h$ are displayed in Figure \ref{fig:contour-of-uh-for-lambda_h_B}. The numerical computation tells that the maximum value of $u_h$ happens on the midpoint of the hypotenuse of $K$. Note that the maximum value of $u_h$ is around $0.95$ while the maximum of its Bernstein coefficients is above $1$. 

\begin{figure}[htp]
\begin{center}
	\includegraphics[width=5.5cm, height=5cm]{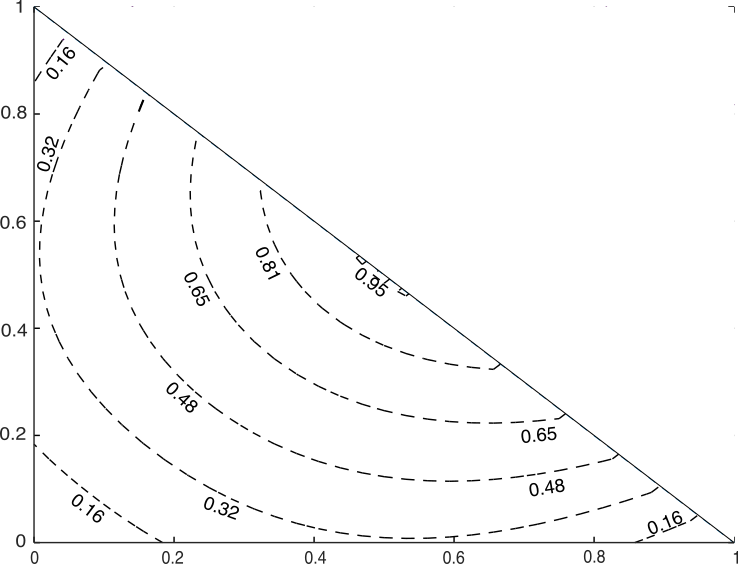}
\end{center}
\caption{\label{fig:contour-of-uh-for-lambda_h_B} The contour lines of the minimizer $u_h$ of (\ref{eq:optimprob2}) for  $K_{1,\pi/2}.$}
\end{figure}

Let us also compare the lower bounds of $\lambda$ obtained through Theorem \ref{th-main} and Corollary \ref{th-cor} for various triangles. Table \ref{tab:comparison} tells that the values obtained using Corollary \ref{th-cor} gives a sharper estimate of $\lambda$.   

Table \ref{tab:upperlowerbounds2} summarizes the results for the lower and upper bounds of the constant for different types of triangle $K_{1,\theta}$ with the mesh size as $h=1/32$ and $h=1/64$. The upper bounds (denoted by $C^L_{ub}$) are obtained through Corollary \ref{th-cor}, while the lower bounds 
(denoted by $C^L_{lb}$) 
are obtained by using high-degree polynomials with degree denoted by $d$. 

Figure \ref{fig:bounds} demonstrates the convergency of the upper and lower bounds of the interpolation error constant as the mesh is refined. It implies that the convergency order of upper bounds depends on the shape of the triangles.
The theoretical analysis on the efficiency and the convergency of the algorithm
in solving the optimization problem is beyond the scope of this paper and will be systematically investigated in our succeeding research.



\begin{table}[h]
    \centering
    \caption{The lower bounds for $\lambda$  through Theorem \ref{th-main} and Corollary \ref{th-cor}.}    \begin{tabular}{|c|c|c|c|c|c|c|}
         \hline
         \rule[-2mm]{0mm}{5mm}{}
         $\theta$ &  \multicolumn{3}{|c|}{$h=1/32$} &
         \multicolumn{3}{c|}{$h=1/64$} \\
         \hline
         \rule[-2mm]{0mm}{5mm}{}
        & $\lambda_{h,B}$ & Thm. 3.1 & Cor. 3.1 & $\lambda_{h,B}$ & Thm. 3.1 & Cor. 3.1 \\
        \hline
        \rule[-2mm]{0mm}{5mm}{}
        $\pi/6$ & $9.8339$ & $8.7356$ & $ 9.8245$ & $9.8925$ & $9.5892$ & $9.8901$  \\
        \hline
        \rule[-2mm]{0mm}{5mm}{}
        $\pi/4$ & $13.517$ & $12.263$ & $13.505$ & $13.574$ & $13.234$ & $13.570$ \\
        \hline
        \rule[-2mm]{0mm}{5mm}{}
        $\pi/3$  & $15.412$ & $14.357$ & $15.397$ & $15.457$ & $15.177$ & $15.454$ \\
        \hline
        \rule[-2mm]{0mm}{5mm}{}
        $\pi/2$  & $5.5988$ & $5.5418$ & $5.5933$ & $5.7812$ & $5.7660$ & $5.7799$ \\
        \hline
        \rule[-2mm]{0mm}{5mm}{}
        $2\pi/3$ & $2.3954$ & $2.3683$ & $2.3930$ & $2.5511$ & $2.5433$ & $2.5504$ \\
        \hline
        \rule[-2mm]{0mm}{5mm}{}
        $3\pi/4$ &  $1.5550$  & $ 1.5369$ & $1.5534$ & $1.6768$ & $1.6715$ & $1.6764$ \\
        \hline
        \rule[-2mm]{0mm}{5mm}{}
        $5\pi/6$ & $0.93778$ & $0.92669$ & $0.93687$ & $1.0212$ & $1.0179$ & $1.0210$ \\
        \hline
    \end{tabular}

    \label{tab:comparison}
\end{table}

\begin{table}[h]
    \centering
    \caption{The lower and upper bounds of $C^L(1,\theta)$ for triangles of different shapes}    \begin{tabular}{|c|c|c||c|c|c|c||c|c|}
         \hline
         \rule[-2mm]{0mm}{5mm}{}
         $\theta$ & \multicolumn{4}{|c|}{$h=1/32$} &\multicolumn{4}{c|}{$h=1/64$} \\
         \hline
         \rule[-2mm]{0mm}{5mm}{}
        &  $d$ & $C^L_{lb}$ & $\lambda_{h,B}$ & $C^L_{ub}$ & $d$ & $C^L_{lb}$ & $\lambda_{h,B}$&  $C^L_{ub}$ \\
        \hline
        \rule[-2mm]{0mm}{5mm}{}
        $\pi/6$ & $9$ & $0.31511$ & $9.8339$ & $0.31904$ & $9$ & $0.31423$ & $9.8925$ & $0.31798$ \\
        \hline
        \rule[-2mm]{0mm}{5mm}{}
        $\pi/4$ & $8$ & $0.26777$ & $13.517$ & $0.27212$ & $8$ & $0.26753$ & $13.574$ & $0.27146$ \\
        \hline
        \rule[-2mm]{0mm}{5mm}{}
        $\pi/3$ & $10$ & $0.25182$ & $15.412$ & $0.25485$ & $10$ & $0.25209$ & $15.457$ & $0.25438$ \\
        \hline
        \rule[-2mm]{0mm}{5mm}{}
        $\pi/2$ & $9$ &$0.40432$ & $5.5988$ & $0.42283$ & $9$ & $0.40419$ & $5.7812$ & $0.41595$ \\
        \hline
        \rule[-2mm]{0mm}{5mm}{}
        $2\pi/3$ & $8$& $0.59964$ & $2.3954$ & $0.64644$ & $8$ & $0.60079$ & $2.5511$ & $0.62617$ \\
        \hline
        \rule[-2mm]{0mm}{5mm}{}
        $3\pi/4$ & $10$ & $0.72146$ & $1.5550$ & $0.80233$ & $10$ & $0.72420$ & $1.6768$ & $0.77234$\\
        \hline
        \rule[-2mm]{0mm}{5mm}{}
        $5\pi/6$ & $8$ & $0.92197$ & $0.93778$ & $1.03314$ & $8$ & $0.92830$ & $1.0212$ & $0.98968$\\
        \hline
    \end{tabular}

    \label{tab:upperlowerbounds2}
\end{table}

\begin{figure}[htp]
\begin{center}
	\includegraphics[scale=0.7]{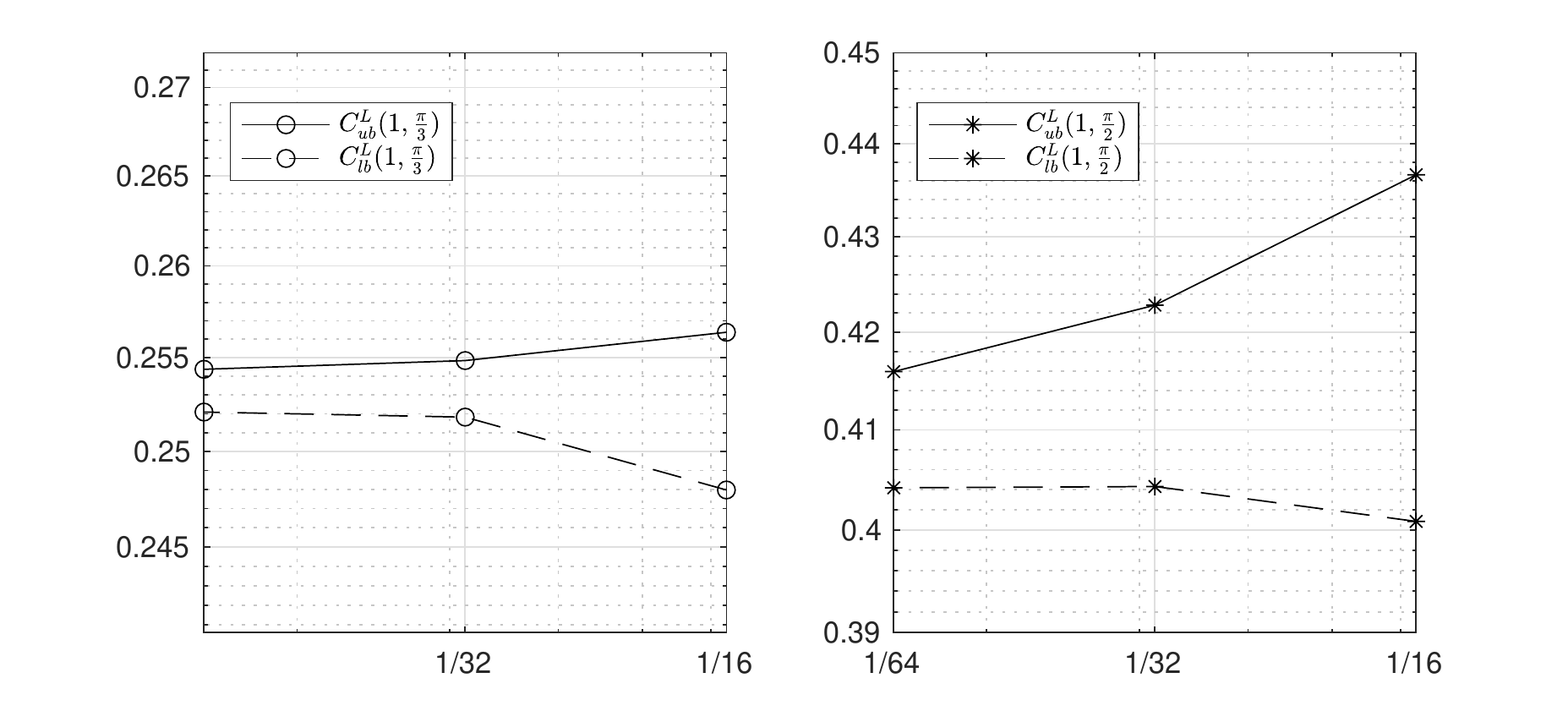}
\end{center}
\caption{\label{fig:bounds} The convergency behaviour of the upper and lower bounds of $C^L(1,\theta)$ for $\theta = \pi/3$ and $\pi/2$.}
\end{figure}

\paragraph{Conclusion} 
In this research, we provide explicit estimates for the $L^\infty$-norm error constant $C^L$ of the linear Lagrange interpolation function over triangular elements.  The formula in Theorem \ref{thm:formula-of-general-bound} provides a bound of $C^L$ that holds for triangle of arbitrary shapes. Theorem \ref{th-main} in \S \ref{sec3} proposes a numerical approach to obtain optimal bounds for the constant $C^L$ over a concrete triangle. The optimization problem corresponding to $C^L$ is novelly solved by utilizing the convex-hull property of Bernstein polynomials. 
 In the near future, the convergency of the numerical approach to solve the optimization problems involving the maximum norm will be systematically considered.


\begin{backmatter}

\section*{Availability of data and material}
An online demo with source codes of the constant evaluation is available at \url{https://ganjin.online/shirley/InterpolationErrorEstimate}.

\section*{Competing interests}
  The authors declare that they have no competing interests.
\section*{Funding} 
The last author is supported by Japan Society for the Promotion of Science: Fund for the Promotion of Joint International Research (Fostering Joint International Research (A)) 20KK0306, 
Grant-in-Aid for Scientific Research (B) 20H01820, 21H00998, and Grant-in-Aid for Scientific Research (C) 18K03411. 
\section*{Authors' contributions}
The first author prepared the manuscript and finished the programming code for the computation examples. The second author prepared the part on solving the objective optimization problem. The last author provided the main idea and advice for this research.
\section*{Acknowledgements}
The authors show great appreciation to Tamaki TANAKA and Syuuji YAMADA from Faculty of Science, Niigata University for their advice on solving the optimization problem in an efficient way, as stated in Lemma 3.1.


\bibliographystyle{bmc-mathphys} 
\bibliography{reference}      

\end{backmatter}
\end{document}